\documentclass[10pt]{article}
\usepackage[total={4.5in,7.125in},centering]{geometry}

\usepackage{amsmath,amssymb,amsthm}

\usepackage{tikz}
\usetikzlibrary{decorations.markings, calc}
\tikzset{vertex/.style={circle, fill=black, draw=black, thick, inner sep=0pt, minimum size=1ex}}
\tikzset{middlearrow/.style={
        decoration={markings,
            mark= at position 0.5 with {\arrow{#1}} ,
        },
        postaction={decorate}
    }
}
\tikzset{oriented/.style={
        decoration={markings, mark=at position \arrowposition with {\arrow[very thick]{stealth}}},
        postaction={decorate}
}}
\tikzset{orientedRev/.style={
        decoration={markings, mark=at position 1-\arrowposition with {\arrowreversed[very thick]{stealth}}},
        postaction={decorate}
}}
\newcommand\arrowposition{0.67}

\newenvironment{mydigraph}[1][]{%
  \begin{tikzpicture}[every edge/.style={draw, semithick}, #1, scale=0.8]
    \foreach \i in {1,...,7} {\coordinate (v\i) at (180-39-360/7*\i : 1);}
}{
    \foreach \i in {1} {\node [vertex, label = {[label distance=-0.5mm]90: $v_0$}] at (v\i) {};}
    \foreach \i in {2} {\node [vertex, label = {[label distance=-1.5mm]30: $v_1$}] at (v\i) {};}
    \foreach \i in {3} {\node [vertex, label = {[label distance=-0.5mm]0: $v_2$}] at (v\i) {};}
    \foreach \i in {4} {\node [vertex, label = {[label distance=-1mm]-15: $v_3$}] at (v\i) {};}
    \foreach \i in {5} {\node [vertex, label = {[label distance=-1.2mm]225: $v_4$}] at (v\i) {};}
    \foreach \i in {6} {\node [vertex, label = {[label distance=-0.5mm]180: $v_5$}] at (v\i) {};}
    \foreach \i in {7} {\node [vertex, label = {[label distance=-1.5mm]120: $v_6$}] at (v\i) {};}
  \end{tikzpicture}}
  
  \newenvironment{mydigraph2}[1][]{%
  \begin{tikzpicture}[every edge/.style={draw, semithick}, #1, scale=0.8]
    \foreach \i in {1,...,7} {\coordinate (v\i) at (180-39-360/7*\i : 1);}
}{
    \foreach \i in {1} {\node [vertex] at (v\i) {};}
    \foreach \i in {2} {\node [vertex] at (v\i) {};}
    \foreach \i in {3} {\node [vertex] at (v\i) {};}
    \foreach \i in {4} {\node [vertex] at (v\i) {};}
    \foreach \i in {5} {\node [vertex] at (v\i) {};}
    \foreach \i in {6} {\node [vertex] at (v\i) {};}
    \foreach \i in {7} {\node [vertex] at (v\i) {};}
  \end{tikzpicture}}

\newtheorem{theorem}{Theorem}[section]
\newtheorem{proposition}[theorem]{Proposition}
\newtheorem{corollary}[theorem]{Corollary}
\newtheorem{lemma}[theorem]{Lemma}

\theoremstyle{definition}
\newtheorem{example}{Example}[section]

\newcommand{\Z}{\ensuremath{\mathbb{Z}}}

\newcommand{\K}{K^*}
\newcommand\lb[1][0]{\ensuremath{$ $}}

\title{Decompositions of Complete Symmetric Directed Graphs into the Oriented Heptagons}

\author{
  U\u gur Odaba\c s\i 
  \\[1em]
  {\normalsize\shortstack[l]{%
   Illinois State University, Normal, IL 61790-4520, USA
  }}
}

\begin{document}

\maketitle
\thispagestyle{empty}

\begin{abstract} \noindent
The complete symmetric directed graph of order $v$, denoted $\K_{v}$, is the directed graph on $v$~vertices that contains both arcs $(x,y)$ and $(y,x)$ for each pair of distinct vertices $x$ and~$y$. For a given directed graph, $D$, the set of all $v$ for which $\K_{v}$ admits a $D$-decomposition is called the spectrum of~$D$. There are 10 non-isomorphic orientations of a $7$-cycle (heptagon). In this paper, we completely settled the spectrum problem for each of the oriented heptagons.
\end{abstract}

%
%
%



\section{Introduction}\label{introduction}

For a graph (or directed graph) $D$, we use $V(D)$ and $E(D)$ to denote the vertex set of $D$ and the edge set (or arc set) of $D$, respectively.
For a simple graph $G$, we use $G^*$ to denote the symmetric digraph with vertex set $V(G^*)=V(G)$ and arc set $E(G^*)= \bigcup_{\{x,y\}\in E(G)} \bigl\{ (x,y), (y,x) \bigr\}$.
Hence, $\K_v$ is the complete symmetric directed graph of order~$v$.
We use $K_{r \times s}$ to denote the complete simple multipartite graph with $r$~parts of size~$s$.
Also, if $a$ and $b$ are integers with $a \leq b$, we let $[a,b]$ denote the set $\{a, a+1, \ldots, b\}$.

A \emph{decomposition} of a directed graph $K$ is a set
$\Delta =\{D_1, D_2, \ldots, D_t\}$ of subgraphs of $K$ such that
each directed edge, or arc, of $K$ appears in exactly one directed graph $D_i \in \Delta$.
If each $D_i$ in $\Delta$ is isomorphic to a given digraph $D$, the decomposition is called a \emph{$D$-decomposition} of~$K$.
A $\{G,H\}$-decomposition of $K$ is defined similarly.
A $D$-decomposition of $K$ is also known as a \emph{$(K, D)$-design}.
The set of all $v$ for which $\K_{v}$ admits a $D$-decomposition is called the \emph{spectrum of~$D$}.

The \emph{reverse orientation} of $D$, denoted Rev($D$), is the digraph with vertex set $V(D)$ and arc set $\bigl\{ (v,u) : (u,v)\in E(D) \bigr\}$.
We note that the existence of a $D$-decomposition of $K$ necessarily implies the existence of a Rev($D$)-decomposition of Rev($K$).
Since $\K_{v}$ is its own reverse orientation, we note that the spectrum of $D$ is equivalent to the spectrum of Rev($D$).

The necessary conditions for a digraph $D$ to decompose $\K_{v}$ include
\begin{enumerate}
  \itemsep1pt \parskip0pt \parsep0pt
  \renewcommand{\theenumi}{(\alph{enumi})}
  \renewcommand{\labelenumi}{\theenumi}
  \item\label{cond:order} $|V(D)|\leq v$,
  \item\label{cond:size} $|E(D)|$ divides $v(v-1)$, and
  \item\label{cond:degree} gcd$\{$outdegree$(x) : x\in V(D)\}$ and gcd$\{$indegree$(x) : x\in V(D)\}$ both divide $(v-1)$.
\end{enumerate}

The spectrum problem for certain subgraphs (both bipartite and nonbipartite) of $\K_{4}$ has already been studied.
When $D$ is a cyclic orientation of $K_{3}$, then a $(\K_{v},D)$-design is known as a Mendelsohn triple system.
The spectrum for Mendelsohn triple systems was found independently by Mendelsohn~\cite{Me} and Bermond~\cite{Be}.
When $D$ is a transitive orientation of $K_{3}$, then a $(\K_{v},D)$-design is known as a transitive triple system.
The spectrum for transitive triple systems was found by Hung and Mendelsohn~\cite{HuMe}.
There are exactly four orientations of a quadrilateral (i.e., a $4$-cycle).
It was shown in~\cite{Sc} that if $D$ is a cyclic orientation of a quadrilateral, then a $(\K_{v},D)$-design exists if and only if $v\equiv 0$ or $1\pmod{4}$ and $v\neq 4$.
The spectrum problem for the remaining three orientations of a quadrilateral were setteled in~\cite{HaWaHe}.
In \cite{AlHeVa}, Alspach et~al.\ showed that $\K_{v}$ can be decomposed into each of the four orientations of a pentagon
(i.e., a $5$-cycle) if and only if $v\equiv 0$ or $1\pmod{5}$.
In \cite{AlGaSaVe}, it is shown that for positive integers $m$ and $v$ with $2\leq m\leq v$ the directed graph $\K_{v}$ can be decomposed into directed cycles (i.e., with all the edges being oriented in the same direction) of length $m$ if and only if $m$ divides the number of arcs in $\K_{v}$ and $(v,m)\notin \bigl\{ (4,4), (6,3), (6,6) \bigr\}$.
Also recently \cite{lambdaC6}, Adams et~al. settled the $\lambda$-fold spectrum problem for all possible orientations of a $6$-cycle.

There are ten non-isomorphic orientations of a heptagon (see Figure~\ref{fig:OrientedC7's}).
The spectrum problem was settled for the directed heptagon ($D_{10}$ in Figure \ref{fig:OrientedC7's}) in \cite{AlGaSaVe}. 

\begin{theorem}[\cite{AlGaSaVe}] \label{thm:Directed6-cycle}
  For integers $v\geq 7$,
  there exists a $D_{10}$-decomposition of $\K_{v}$
  if and only if
  $v(v-1)\equiv 0\pmod{7}$.
\end{theorem}

In this work, we completely settle the spectrum problem for the remaining nine non-isomorphic oriented heptagons
(i.e., $D_i$ for $i\in [1,9]$ as seen in Figure~\ref{fig:OrientedC7's}).

If $D$ is an oriented heptagon and if there exists a  $(\K_{v},D)-$design, then we must have $7|v(v-1)$. Moreover, if $D$ is an oriented odd cycle, then gcd$\{$outdegree $(x)\colon\ x\in V(D)\}=$gcd$\{$indegree$(x)\colon\ x\in V(D)\}=1$. Thus from the necessary conditions established above we have the following.

\begin{lemma} \label{lem:MainNecessaryConditions}
  Let $D \in \{D_1,D_2,\ldots, D_9\}$ and let $v\geq7$ be a positive integer.
  There exists a $D$-decomposition of\/ $\K_v$
  only if $v(v-1) \equiv 0\pmod{7}$.
\end{lemma}
The remainder of this paper is dedicated to showing the existence of the decompositions in questions in order to establish sufficiency of these necessary conditions.
Henceforth, each of the graphs $D_i$ with $i\in [1,9]$ in Figure~\ref{fig:OrientedC7's}, with vertices labeled as in the figure, will be represented by ${D_i}[v_0, v_1,\ldots, v_6]$. For example, ${D_3}[v_0, v_1,\ldots, v_6]$  is the graph with vertex set $\{v_0,v_1,v_2,v_3,v_4,v_5,v_6\}$  and arc set $\bigl\{(v_1,v_0),(v_1,v_2),(v_3,v_2),(v_3,v_4),\allowbreak (v_4,v_5),(v_5,v_6),(v_6,v_0) \bigr\}$.

\begin{figure}
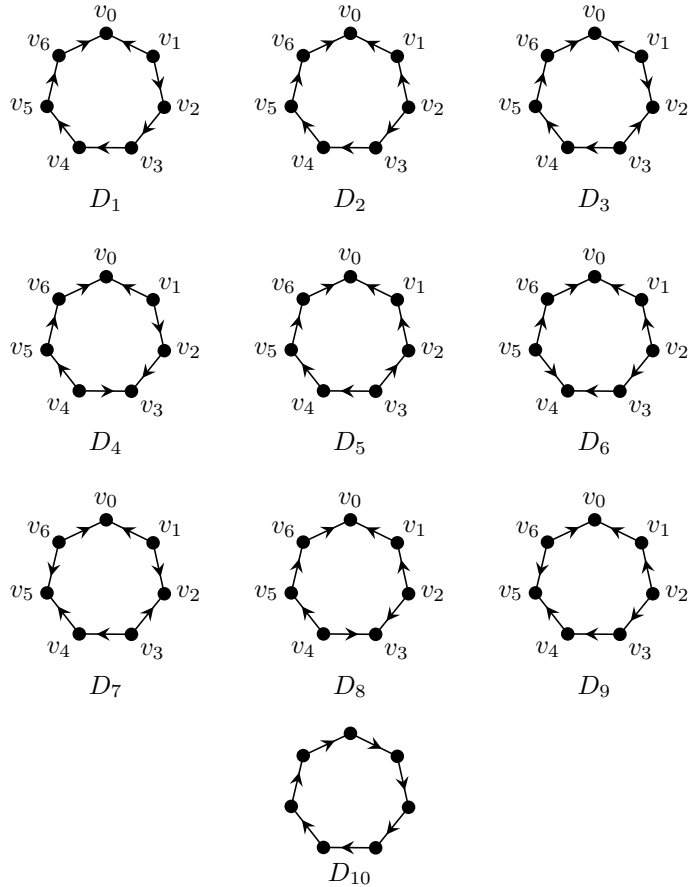

  \centering
  \addtolength\lineskip{1em}
  \shortstack{
    \begin{mydigraph}
      \path
        (v2) edge[oriented] (v1)
        (v2) edge[oriented] (v3)
        (v3) edge[oriented] (v4)
        (v4) edge[oriented] (v5)
        (v5) edge[oriented] (v6)
        (v6) edge[oriented] (v7)
        (v7) edge[oriented] (v1);
    \end{mydigraph}
    \\ $D_1$}
  \quad
  \shortstack{
    \begin{mydigraph}
      \path
        (v2) edge[oriented] (v1)
        (v3) edge[oriented] (v2)
        (v3) edge[oriented] (v4)
        (v4) edge[oriented] (v5)
        (v5) edge[oriented] (v6)
        (v6) edge[oriented] (v7)
        (v7) edge[oriented] (v1);
    \end{mydigraph}
    \\ $D_2$}
  \quad
  \shortstack{
    \begin{mydigraph}
      \path
        (v2) edge[oriented] (v1)
        (v2) edge[oriented] (v3)
        (v4) edge[oriented] (v3)
        (v4) edge[oriented] (v5)
        (v5) edge[oriented] (v6)
        (v6) edge[oriented] (v7)
        (v7) edge[oriented] (v1);
    \end{mydigraph}
    \\ $D_3$}
  \quad
  \shortstack{
    \begin{mydigraph}
      \path
        (v2) edge[oriented] (v1)
        (v2) edge[oriented] (v3)
        (v3) edge[oriented] (v4)
        (v5) edge[oriented] (v4)
        (v5) edge[oriented] (v6)
        (v6) edge[oriented] (v7)
        (v7) edge[oriented] (v1);
    \end{mydigraph}
    \\ $D_4$}
  \quad
  \shortstack{
    \begin{mydigraph}
      \path
        (v2) edge[oriented] (v1)
        (v3) edge[oriented] (v2)
        (v4) edge[oriented] (v3)
        (v4) edge[oriented] (v5)
        (v5) edge[oriented] (v6)
        (v6) edge[oriented] (v7)
        (v7) edge[oriented] (v1);
    \end{mydigraph}
    \\ $D_5$}
  \quad
  \shortstack{
    \begin{mydigraph}
      \path
        (v2) edge[oriented] (v1)
        (v3) edge[oriented] (v2)
        (v3) edge[oriented] (v4)
        (v4) edge[oriented] (v5)
        (v6) edge[oriented] (v5)
        (v6) edge[oriented] (v7)
        (v7) edge[oriented] (v1);
    \end{mydigraph}
    \\ $D_6$}
  \quad
  \shortstack{
    \begin{mydigraph}
      \path
        (v2) edge[oriented] (v1)
        (v2) edge[oriented] (v3)
        (v4) edge[oriented] (v3)
        (v4) edge[oriented] (v5)
        (v5) edge[oriented] (v6)
        (v7) edge[oriented] (v6)
        (v7) edge[oriented] (v1);
    \end{mydigraph}
    \\ $D_7$}
  \quad
    \shortstack{
    \begin{mydigraph}
      \path
        (v2) edge[oriented] (v1)
        (v3) edge[oriented] (v2)
        (v3) edge[oriented] (v4)
        (v5) edge[oriented] (v4)
        (v5) edge[oriented] (v6)
        (v6) edge[oriented] (v7)
        (v7) edge[oriented] (v1);
    \end{mydigraph}
    \\ $D_8$}
  \quad
  \shortstack{
    \begin{mydigraph}
      \path
        (v2) edge[oriented] (v1)
        (v3) edge[oriented] (v2)
        (v3) edge[oriented] (v4)
        (v4) edge[oriented] (v5)
        (v5) edge[oriented] (v6)
        (v7) edge[oriented] (v6)
        (v7) edge[oriented] (v1);
    \end{mydigraph}
    \\ $D_9$}
      \quad
      \shortstack{
    \begin{mydigraph2}
      \path
        (v1) edge[oriented] (v2)
        (v2) edge[oriented] (v3)
        (v3) edge[oriented] (v4)
        (v4) edge[oriented] (v5)
        (v5) edge[oriented] (v6)
        (v6) edge[oriented] (v7)
        (v7) edge[oriented] (v1);
    \end{mydigraph2}
    \\ $D_{10}$}
  \caption{The ten orientated heptagons.}
  \label{fig:OrientedC7's}
\end{figure}

The following result of Alspach and Gavlas, proves the existence of $7$-cycle decompositions of complete graphs.

\begin{theorem}[\cite{AlGa}] \label{C7Kv}
 Let $v\geq 7$ be an integer. There exists a $C_7$-decom\-po\-si\-tion of $K_{v}$ if and only if $v \equiv 1$ or $7 \pmod{14}$.
\end{theorem}

In \cite{liu,liu2} Liu construct cycle decompositions of complete equipartite graphs. 
Although Liu’s result concerned resolvable decompositions, we omit this detail
since we will only make use of the following particular instances of his result.

\begin{proposition}[\cite{liu,liu2}] \label{C7Knby7}
Let $n\geq 3$ be an odd integer. There exists a $C_7$-de\-com\-po\-si\-tion of $K_{n \times 7}$.
\end{proposition}

For the orientations of a heptagon that are isomorphic to their own reverse,
we have the following as a direct result of Theorem~\ref{C7Kv}.

\begin{lemma} \label{reduction}
If $K_v$ ($K_{r \times s}$) has a $C_7$-decomposition, then $\K_v$ ($\K_{r \times s}$) has a $D_i$-de\-com\-po\-si\-tion for $i\in [1,7]$. 
\end{lemma}

Theorem \ref{C7Kv} together with Lemma \ref{reduction} give the following result.

\begin{corollary} \label{1,7(mod14)forRev}
If $v \equiv 1$ or $7 \pmod{14}$, then $\K_v$ has a $D_i$-decomposition for $i\in [1,7]$. 
\end{corollary}

Similarly, Proposition \ref{C7Knby7} together with Lemma \ref{reduction} give the following result.

\begin{corollary} \label{Knby7}
For all odd integer $n\geq 3$, there exists a $D_i$-decomposition of $\K_{n \times 7}$  for $i\in [1,7]$. 
\end{corollary}

We also make use of the following results in the next section. All of these results can be found in the \emph{Handbook of Combinatorial Designs} \cite{CoDi}.

\begin{theorem} \label{K3K5-odd}
 Let $v\geq 3$ be an odd integer. There exists a $\{K_3,K_5\}$-de\-com\-po\-si\-tion of $K_{v}$.
\end{theorem}

\begin{theorem} \label{K3K5-even}
 Let $v\geq 6$ be an even integer. There exists a $\{K_3,K_5\}$-de\-com\-po\-si\-tion of $K_{v}-I$ where $I$ is a $1$-factor of $K_{v}$.
\end{theorem}

\section{Examples of Small Designs}

We first present several $D_i$-decompositions of various graphs for $i \in [1,8]$.
Beyond establishing existence of necessary base cases, 
these decompositions are used extensively in the general constructions seen in Section~\ref{Constructions}.

If $i,v_0,v_1,\ldots,v_6$ are integers and $D\in\{D_1, D_2,\dots, D_8\}$,
we define $D[v_0, v_1,\allowbreak \dots, v_6]+i$ to indicate $D[v_0+i, v_1+i,\dots, v_6+i]$.

\begin{example} \label{exK7}
  Let $V\bigl( \K_7 \bigr) =\Z_{6}\cup \{\infty\}$ and let
\begin{align*}
   \Delta_8 = \{& D_8 [0,1,3,4,2,5,\infty] +i : i \in \Z_{6} \}. 
\end{align*}   
Then $\Delta_8$ is a $D_8$-decomposition of $\K_{7}$.
\end{example}

\begin{example} \label{exK8}
  Let $V\bigl( \K_8 \bigr) =\Z_{8}$ and let
\begin{align*}
	\Delta_1 = \{& D_1[0,2,3,7,4,6,5]+i: i \in \Z_{8}\}, \\
	\Delta_2 = \{& D_2[0,5,6,3,1,2,4]+i: i \in \Z_{8}\}, \\
	\Delta_3 = \{& D_3[0,2,3,7,4,6,1]+i: i \in \Z_{8}\}, \\
	\Delta_4 = \{& D_4[0,1,3,7,4,5,2]+i: i \in \Z_{8}\}, \\
	\Delta_5 = \{& D_5[0,1,3,2,5,7,4]+i: i \in \Z_{8}\}, \\
	\Delta_6 = \{& D_6[0,1,3,4,7,2,6]+i: i \in \Z_{8}\}, \\
	\Delta_7 = \{& D_7[0,1,2,4,7,3,6]+i: i \in \Z_{8}\}, \\
	\Delta_8 = \{& D_8[0,1,3,7,2,4,5]+i: i \in \Z_{8}\}. 
\end{align*}   
Then $\Delta_i$ is a $D_i$-decomposition of $\K_{8}$ for $i\in [1,8]$.
\end{example}

\begin{example} \label{exK14}
  Let $V\bigl( \K_{14} \bigr) =\Z_{13}\cup \{\infty\}$ and let
\begin{align*}
    \Delta_1 = \{& \{D_1 [0,1,12,2,11,3,9]\cup D_1 [0,3,4,12,1,8,\infty]\} +i : i \in \Z_{13} \},\\
    \Delta_2 = \{& \{D_2 [0,1,12,2,11,3,9]\cup D_2 [0,2,5,6,1,8,\infty]\} +i : i \in \Z_{13} \},\\  
    \Delta_3 = \{& \{D_3 [0,1,12,2,11,3,9]\cup D_3 [0,5,6,4,7,1,\infty]\} +i : i \in \Z_{13} \},\\  
    \Delta_4 = \{& \{D_4 [0,1,12,2,11,3,5]\cup D_4 [0,3,4,10,1,8,\infty]\} +i : i \in \Z_{13} \},\\  
    \Delta_5 = \{& \{D_5 [0,1,12,2,11,3,9]\cup D_5 [0,2,1,6,9,3,\infty]\} +i : i \in \Z_{13} \},\\  
    \Delta_6 = \{& \{D_6 [0,1,12,2,11,3,9]\cup D_6 [0,2,5,6,11,4,\infty]\} +i : i \in \Z_{13} \},\\
    \Delta_7 = \{& \{D_7 [0,1,12,2,11,3,9]\cup D_7 [0,5,6,\infty,7,1,11]\} +i : i \in \Z_{13} \},\\  
    \Delta_8 = \{& \{D_8 [0,1,12,2,11,3,4]\cup D_8 [0,2,5,11,3,10,\infty]\} +i : i \in \Z_{13} \}. 
\end{align*}   
Then $\Delta_i$ is a $D_i$-decomposition of $\K_{14}$ for $i\in [1,8]$.
\end{example}

\begin{example} \label{exK15}
  Let $V\bigl( \K_{15} \bigr) =\Z_{15}$ and let
\begin{align*}
   \Delta_8 = \{& D_8 [0,1,3,5,2,14,4]\cup D_8 [0,14,10,1,6,13,7] \} +i : i \in \Z_{15} \}. 
\end{align*}   
Then $\Delta_8$ is a $D_8$-decomposition of $\K_{15}$.
\end{example}

\begin{example} \label{exK28}
  Let $V\bigl( \K_{28} \bigr) =\Z_{27}\cup \{\infty\}$ and let
\begin{align*}
   \Delta_1 =  \{& \{ D_1[0,1,26,2,25,3,23]\cup D_1[0,3,25,4,23,5,26]\cup \\
	& D_1[0,9,16,6,14,3,13]\cup D_1[0,12,14,26,10,23,\infty]\} +i : i \in \Z_{27}\},\\
	\Delta_2 = \{& \{ D_2[0,1,26,2,25,3,23]\cup D_2[0,3,19,4,23,5,26]\cup\\
	& D_2[0,5,16,6,14,3,13]\cup D_2[0,2,14,20,6,13,\infty]\} +i : i \in \Z_{27}\},\\ 
	\Delta_3 = \{& \{ D_3[0,1,26,2,25,3,23]\cup D_3[0,16,19,4,23,5,26]\cup \\
	& D_3[0,5,12,6,14,3,13]\cup D_3[0,9,11,21,6,19,\infty]\} +i : i \in \Z_{27}\},\\
	\Delta_4 = \{& \{ D_4[0,1,26,2,24,3,25]\cup D_4[0,3,23,4,22,5,20]\cup \\
	& D_4[0,9,5,16,4,17,10]\cup D_4[0,8,9,13,26,15,\infty]\} +i : i \in \Z_{27}\},\\ 
	\Delta_5 = \{& \{ D_5[0,1,26,2,25,3,23]\cup D_5[0,2,14,4,23,5,26]\cup \\
	& D_5[0,5,16,26,7,10,21]\cup D_5[0,11,4,17,3,15,\infty]\} +i : i \in \Z_{27}\},\\
	\Delta_6 = \{& \{ D_6[0,1,26,2,25,3,23]\cup D_6[0,2,5,26,18,1,11]\cup \\
	& D_6[0,9,21,26,5,18,19]\cup D_6[0,14,2,13,22,15,\infty]\} +i : i \in \Z_{27}\},\\
	\Delta_7 = \{& \{ D_7[0,1,26,2,25,3,24]\cup D_7[0,5,26,6,25,16,15],\\
	& D_7[0,10,26,11,18,1,14]\cup D_7[0,16,18,9,\infty,23,19]\} +i : i \in \Z_{27}\},\\
	\Delta_8 = \{& \{ D_8[0,1,26,2,25,3,21]\cup D_8[0,2,5,26,3,25,18],\\
	& D_8[0,8,18,7,19,2,16]\cup D_8[0,15,2,3,23,4,\infty]\} +i : i \in \Z_{27}\}.
\end{align*}   
Then $\Delta_i$ is a $D_i$-decomposition of $\K_{28}$ for $i\in [1,8]$.
\end{example}

\begin{example} \label{exK29}
  Let $V\bigl( \K_{29} \bigr) =\Z_{29}$ and let
\begin{align*}
	\Delta_8 = \{& \{ D_8[0,3,21,4,20,5,28]\cup D_8[0,27,5,26,16,25,21],\\
	& D_8[0,26,9,27,8,28,23]\cup D_8[0,22,23,21,5,20,24]\} +i : i \in \Z_{29}\}.
\end{align*}   
Then $\Delta_8$ is a $D_8$-decomposition of $\K_{29}$.
\end{example}

\begin{example} \label{exK3by7}
 Let $V\bigl(\K_{3\times 7}\bigr) = \Z_{21}$ with vertex partition $\{V_i: i\in \Z_3\}$, where $V_i=\{j\in \Z_{21}:j\equiv i\pmod{3}\}$. Let
\begin{align*}
	\Delta_8 = \{& \{ D_8[0,1,5,12,4,17,10]\cup D_8[0,2,7,12,8,18,20]\} +i : i \in \Z_{21}\}.
\end{align*}   
Then $\Delta_8$ is a $D_8$-decomposition of $\K_{3\times 7}$.
\end{example}

\begin{example} \label{exK5by7}
 Let $V\bigl(\K_{5\times 7}\bigr) = \Z_{35}$ with vertex partition $\{V_i: i\in \Z_5\}$, where $V_i=\{j\in \Z_{35}:j\equiv i\pmod{5}\}$. Let
\begin{align*}
	\Delta_8 = \{& \{ D_8[0,1,3,19,2,16,17]\cup D_8[0,3,7,18,6,8,16]\cup \\
	& D_8[0,6,13,19,12,1,14]\cup D_8[0,8,17,20,16,4,13]\} +i : i \in \Z_{35}\}.
\end{align*}   
Then $\Delta_8$ is a $D_8$-decomposition of $\K_{5\times 7}$.
\end{example}
\section{General Constructions}\label{Constructions}

For two edge-disjoint graphs (or digraphs) $G$ and $H$, we use $G\cup H$ to denote the graph (or digraph) with vertex set $V(G)\cup V(H)$ and edge (or arc) set $E(G)\cup E(H)$.
Furthermore, given a positive integer $x$, we use $xG$ to denote the vertex-disjoint union of $x$~copies of $G$. 

We now give our constructions for decompositions of $\K_v$ in Lemmas \ref{0mod14}, \ref{1,7(mod14)forRev}, and \ref{8mod14} which cover values of $v$ working modulo~$7$.
The main result is summarized in Theorem~\ref{thm:MainResult}.

\begin{lemma} \label{(2x)times7}
For every integer $x\geq 3$ and each $i \in [1,8]$, there exists a $D_i$-decomposition of $\K_{(2x) \times 7}-x\K_{7,7}$.
\end{lemma}

\begin{proof} 
By Theorem~\ref{K3K5-even}, there exists a $\{\K_3,\K_5\}$-decomposition of $\K_{2x}-I^*$.
Replacing each vertex of $\K_{2x}$ by a set of 7 vertices and each edge of $\K_{2x}-I^*$ by a copy of $\K_{7,7}$ gives a $\{\K_{3\times 7},\K_{5\times 7}\}$-decomposition of $\K_{(2x) \times 7}-x\K_{7,7}$.
By Corollary~\ref{C7Knby7}, there exists a $D_i$-decomposition of $\K_{3 \times 7}$ and $\K_{5 \times 7}$ for $i\in [1,7]$. 
Also by Examples~\ref{exK3by7} and \ref{exK5by7}, there exist $D_8$\nobreakdash-decompositions of $\K_{3\times 7}$ and $\K_{5\times 7}$, respectively.
Thus the result now follows.
\end{proof}

\begin{lemma} \label{0mod14}
Let $D \in \{D_1, D_2,\ldots, D_8\}$. There exists a $D$-de\-com\-po\-si\-tion of $\K_v$ for all $v\equiv 0  \pmod{14}$.
\end{lemma}

\begin{proof}
Let $D\in\{D_1, D_2,\ldots, D_8\}$.
For $v=14$ and $v=28$, the results follow from Example~\ref{exK14} and~\ref{exK28}, respectively. 
For the remainder of the proof, we let $v=14x$ for some integer $x\geq3$.

Let $H$ be the complete symmetric directed graph with vertex set $H_1 \cup H_2 \cup \ldots \cup H_{2x}$ with $|H_j|=7$ for each $j\in [1,2x]$.
For each $i\in[1,8]$ and $j\in[1,x]$, let $B_{i,j}$ be a $D_i$-decomposition of $\K_{14}$ with vertex set $H_{2j-1}\cup H_{2j}$.
By Lemma~\ref{(2x)times7}, there is a $D_i$-decomposition $B'_i$ of  $\K_{(2x) \times 7}-x\K_{7,7}$ with vertex set $H_1 \cup H_2 \cup \ldots \cup H_{2x}$, edge set $E(H)-\bigcup_{1\leq j\leq x}\{E(H_{2j-1}\cup H_{2j})\}$, and the obvious vertex partition.
Then $(V,B_i)$ where $V=V(H)$ and $B_i=B'_i \cup B_{i,1}\cup B_{i,2}\cup \ldots \cup B_{i,x}$ is a $D_i$-decomposition of $\K_{14x}$.
\end{proof}

\begin{lemma} \label{1mod14}
Let $D \in \{D_1, D_2,\ldots, D_8\}$. There exists a $D$-de\-com\-po\-si\-tion of $\K_v$ for all $v\equiv 1 \pmod{14}$.
\end{lemma}

\begin{proof}
For $D \in \{D_1, D_2,\ldots, D_7\}$, the result follows from Corollary \ref{1,7(mod14)forRev}. Thus it remains to prove the lemma for $D=D_8$. 

Let  $D=D_8$. In the case $v=15$ and $v=29$, the results follow from Examples \ref{exK15} and \ref{exK29}, respectively. For the remainder of the proof, we let $v=14x+1$ for some integer $x\geq3$.

Let $H$ be the complete symmetric directed graph with vertex set $H_1 \cup H_2 \cup \ldots \cup H_{2x}\cup\{\infty\}$ with $|H_j|=7$ for each $j\in [1,2x]$.
For each $j\in[1,x]$, let $B_{j}$ be a $D_8$-decomposition of $\K_{15}$ with vertex set $H_{2j-1}\cup H_{2j}\cup\{\infty\}$.
By Lemma~\ref{(2x)times7}, there is a $D_8$-decomposition $B'$ of  $\K_{(2x) \times 7}-x\K_{7,7}$ with vertex set $H_1 \cup H_2 \cup \ldots \cup H_{2x}$, edge set $E(H)-\bigcup_{1\leq j\leq x}\{E(H_{2j-1}\cup H_{2j})\}$, and the obvious vertex partition.
Then $(V,B_i)$ where $V=V(H)$ and $B_i=B' \cup B_{1}\cup B_{2}\cup \ldots \cup B_{x}$ is a $D_8$-decomposition of $\K_{14x+1}$.
\end{proof}

\begin{lemma} \label{(2x+1)times7}
For every integer $x\geq 3$ and each $i \in [1,8]$, there exists a $D_i$-decomposition of $\K_{(2x+1) \times 7}$.
\end{lemma}

\begin{proof} 
By Theorem~\ref{K3K5-odd}, there exists a $\{\K_3,\K_5\}$-decomposition of $\K_{2x+1}$.
Replacing each vertex of $\K_{2x+1}$ by a set of 7 vertices and each edge of $\K_{2x+1}$ by a copy of $\K_{7,7}$ gives a $\{\K_{3\times 7},\K_{5\times 7}\}$-decomposition of $\K_{(2x+1) \times 7}$.
By Corollary~\ref{C7Knby7}, there exists a $D_i$-decomposition of $\K_{3 \times 7}$ and $\K_{5 \times 7}$ for $i\in [1,7]$. 
Also by Examples~\ref{exK3by7} and \ref{exK5by7}, there exist $D_8$\nobreakdash-decompositions of $\K_{3\times 7}$ and $\K_{5\times 7}$, respectively.
Thus the result now follows.
\end{proof}

\begin{lemma} \label{7mod14}
Let $D \in \{D_1, D_2,\ldots, D_8\}$. There exists a $D$-de\-com\-po\-si\-tion of $\K_v$ for all $v\equiv 7 \pmod{14}$.
\end{lemma}

\begin{proof}
For $D \in \{D_1, D_2,\ldots, D_7\}$, the result follows from Corollary \ref{1,7(mod14)forRev}. Thus it remains to prove the lemma for $D=D_8$. 

Let  $D=D_8$. In the case $v=7$, the result follows from Examples \ref{exK7}. For the remainder of the proof, we let $v=14x+7$ for some integer $x>1$.

Let $H$ be the complete symmetric directed graph with vertex set $H_1 \cup H_2 \cup \ldots \cup H_{2x+1}$ with $|H_j|=7$ for each $j\in [1,2x+1]$.
For each $j\in[1,2x+1]$, let $B_{j}$ be a $D_8$-decomposition of $\K_{7}$ with vertex set $H_{j}$.
By Lemma~\ref{(2x+1)times7}, there is a $D_8$-decomposition $B'$ of  $\K_{(2x+1) \times 7}$ with vertex set $H_1 \cup H_2 \cup \ldots \cup H_{2x+1}$ and the obvious vertex partition.
Then $(V,B_i)$ where $V=V(H)$ and $B_i=B' \cup B_{1}\cup B_{2}\cup \ldots \cup B_{2x+1}$ is a $D_8$-decomposition of $\K_{14x+7}$.
\end{proof}

\begin{lemma} \label{8mod14}
Let $D \in \{D_1, D_2,\ldots, D_8\}$. There exists a $D$-de\-com\-po\-si\-tion of $\K_v$ for all $v\equiv 8  \pmod{14}$.
\end{lemma}

\begin{proof}
Let $D\in\{D_1, D_2,\ldots, D_8\}$.
For $v=8$, the result follows from Example~\ref{exK8}.
For the remainder of the proof, we let $v=14x+8$ for some integer $x>1$.

Let $H$ be the complete symmetric directed graph with vertex set $H_1 \cup H_2 \cup \ldots \cup H_{2x+1}\cup \{\infty\}$ with $|H_j|=7$ for each $j\in [1,2x+1]$.
For each $i\in[1,8]$ and $j\in[1,2x+1]$, let $B_{i,j}$ be a $D_i$-decomposition of $\K_{8}$ with vertex set $H_{j}\cup \{\infty\}$.
By Lemma~\ref{(2x+1)times7}, there is a $D_i$-decomposition $B'_i$ of  $\K_{(2x+1) \times 7}$ with vertex set $H_1 \cup H_2 \cup \ldots \cup H_{2x+1}$ and the obvious vertex partition.
Then $(V,B_i)$ where $V=V(H)$ and $B_i=B'_i \cup B_{i,1}\cup B_{i,2}\cup \ldots \cup B_{i,2x+1}$ is a $D_i$-decomposition of $\K_{14x+8}$.
\end{proof}

Since $D_9$ is the reverse orientation of $D_8$,  existence of a $(\K_v,D_9)$-design is equivalent to existence of a $(\K_v,D_8)$-design. Thus, combining the previous results from Lemmas~\ref{0mod14},~\ref{1mod14},~\ref{7mod14}, and~\ref{8mod14}, we obtain the following necessary and sufficient conditions for the existence of a decomposition of $\K_v$ into the orientated heptagons.

\begin{theorem} \label{thm:MainResult}
 Let $D$ be an orientated heptagon. There exists a $D$-de\-com\-po\-si\-tion of $\K_v$ if and only if $v \equiv 0$ or $1 \pmod{7}$.
\end{theorem} 

\section{Acknowledgements}
The author wish to thank Saad El-Zanati for helpful suggestions during the course of this work.

\end{document}